\documentclass[12pt]{amsart}

\usepackage{amssymb}
\usepackage{verbatim}
\usepackage[toc,page]{appendix}
\usepackage{mathrsfs}

\usepackage{mathtools}

\newtheorem{thm}{Theorem}[section]

\newtheorem{lem}[thm]{Lemma}

\theoremstyle{definition}

\theoremstyle{remark}

\newtheorem*{rem}{Remark}

\numberwithin{equation}{section}

\newcommand{\kk}{\mathbf{k}}
\newcommand{\HH}{\mathbf{H}}
\newcommand{\M}{\mathcal{M}}
\newcommand{\Z}{\mathbf{Z}}
\newcommand{\Q}{\mathbf{Q}}
\newcommand{\R}{\mathbf{R}}
\newcommand{\C}{\mathbf{C}}
\newcommand{\sap}{\mathbf{sp}}
\providecommand{\sym}{\operatorname{sym}}
\providecommand{\Li}{\operatorname{Li}}

\newcommand{\Mod}[1]{\ (\textup{mod}\ #1)}

\mathtoolsset{showonlyrefs}

\begin{document}

\title[]{
The second moment of symmetric square $L$-functions over Gaussian integers}

\begin{abstract}
We prove a new upper bound on the second moment of Maass form symmetric square $L$-functions defined over Gaussian integers. Combining this estimate with the recent result of Balog-Biro-Cherubini-Laaksonen, we improve the error term in the prime geodesic theorem for the Picard manifold.
\end{abstract}

\author{Olga  Balkanova}
\address{Steklov Mathematical Institute of Russian Academy of Sciences, 8 Gubkina st., Moscow, 119991, Russia}
\email{balkanova@mi-ras.ru}
\author{Dmitry  Frolenkov}
\address{Steklov Mathematical Institute of Russian Academy of Sciences, 8 Gubkina st., Moscow, 119991, Russia}
\email{frolenkov@mi-ras.ru}


\maketitle


\section{Introduction}
Consider the three dimensional hyperbolic space
\begin{equation}
\mathbf{H}^3=\left\{ (z,r); \quad z=x+iy\in \C; \quad r>0\right\}
\end{equation}
and the Picard group defined over Gaussian integers
\begin{equation}
\Gamma=\mathbf{PSL}(2,\mathbf{Z}[i]).
\end{equation}
The prime geodesic theorem for the Picard manifold $\Gamma \setminus \mathbf{H}^3$ provides an asymptotic formula for the function
$\pi_{\Gamma}(X)$, which counts the number of primitive hyperbolic or loxodromic elements in $\Gamma$ with norm less than or equal to $X$. In 1983 Sarnak \cite{Sarnak} proved that
\begin{equation}\label{error:Sarnak}
\pi_{\Gamma}(X)=\Li(X^2)+O(X^{5/3+\epsilon}).
\end{equation}
Since then the function  $\pi_{\Gamma}(X)$ was intensively studied, see \cite{Koyama}, \cite{Nak}, \cite{BCCFL}, \cite{BF2}. The currently best know result due to Balog-Biro-Cherubini-Laaksonen \cite{BBCL}
states that the error term in \eqref{error:Sarnak} can be replaced by 
\begin{equation}\label{eqst:prev}
O(X^{3/2+4\theta/7+\epsilon}),
\end{equation} 
where $\theta=1/6$ (see \cite{N}) is the best known subconvexity exponent  for quadratic Dirichlet $L$-functions defined over Gaussian integers.
In the current paper we improve this result further.

\begin{thm}\label{thm:main} For any $\epsilon>0$
\begin{equation}\label{eq:maineq}
\pi_{\Gamma}(X)=\Li(X^2)+O(X^{\frac{3}{2}+\frac{32\theta^2+28\theta-1}{46+40\theta}+\epsilon}).
\end{equation}
\end{thm}
\begin{rem}
Substituting $\theta=1/6$, we obtain that the error term in \eqref{eq:maineq} is $$O(X^{3/2+41/474+\epsilon})=O(X^{1.586\ldots}),$$ while \eqref{eqst:prev} is equal to $$O(X^{3/2+2/21+\epsilon})=O(X^{1.595\ldots}).$$
\end{rem}

Theorem \ref{thm:main} is a direct consequence of the following estimate on the second moment of Maass form symmetric square $L$-functions defined over Gaussian integers.
\begin{thm}\label{thm:2moment}  For $s=1/2+it$,  $|t|\ll T^{\epsilon}$ we have
\begin{equation}\label{symsquare estimate0}
\sum_{T<r_j<2T}\alpha_{j}|L(\sym^2 u_{j},s)|^2
\ll T^{3+4\theta+\epsilon}.
\end{equation}
\end{thm}

Theorem \ref{thm:2moment}  improves \cite[Theorem 3.3]{BCCFL}, where it was shown that the left-hand side of  \eqref{symsquare estimate0}  can be bounded by $T^{4+\epsilon}$.
The main new ingredient that leads to the improvement is  a more careful treatment of sums of Kloosterman sums.

The strategy for proving  \cite[Theorem 3.3]{BCCFL} consists in using an approximate functional equation for both L-functions followed by the application of the Kuznetsov trace formula.
This results in expressions containing sums of Kloosterman sums multiplied by some complicated weight function. Then \cite[Theorem 3.3]{BCCFL} is proved by estimating the Kloosterman sums using Weil's bound and analyzing the weight function thoroughly.

The proof of Theorem \ref{thm:2moment} is quite different. As the first step we apply an approximate functional equation only for one $L$-function, which reduces our problem to the investigation of the first twisted moment of symmetric square $L$-functions. For this moment we prove an explicit formula which is similar to the one derived in \cite{BF2}. The advantage of such hybrid approach is that it allows us to evaluate sums of Kloosterman sums by replacing them with sums of Zagier $L$-series weighted by a double integral of the Gauss hypergeometric function. The Zagier $L$-series can be estimated using the subconvexity result of Nelson \cite{N}. Consequently, the main difficulty of our approach is the analysis of the sums of the weight function given by the following expression
\begin{equation}\label{eq:doublesumint}
\sum_{|l|\ll T}\sum_{|n|\ll |l| }|n^2-4l^2|^{2\theta}\int_{0}^{1}\int_{T}^{2T}F\left( 1-s-ir,1-s+ir,1;-x_{\pm}(n/l,y)\right)\frac{r^2dr dy}{(1-y^2)^{3/2-s}},
\end{equation}
where  for $\vartheta=\arg(z)$ we have
\begin{equation}\label{x def22}
x_{\pm}(z,y)=\frac{f_{\pm}(z,y)}{1-y^2},\quad f_{\pm}(z,y)=y^2\pm|z|y\cos\vartheta+|z/2|^2.
\end{equation}
When $s=1/2$, the asymptotic formula for the hypergeometric function in \eqref{eq:doublesumint} as $r\rightarrow \infty$  was proved by Jones \cite{Jones} and Farid Khwaja-Olde Daalhuis \cite{KD2}.
However, for the application to the prime geodesic theorem, it is required to consider $s=1/2+it$. For this reason, we prove a uniform version of \cite[Theorem 3.1]{KD2}.
Applying the resulting asymptotic formula for the hypergeometric function in \eqref{eq:doublesumint} and evaluating the integral over $r$, we show that the contribution of the summands with $|x_{\pm}(n/l,y)|\gg T^{\epsilon-2}$ in  \eqref{eq:doublesumint} is negligible.
The final result comes from the opposite case: $|x_{\pm}(n/l,y)|\ll T^{\epsilon-2}$. In this case the hypergeometric function is approximately $1$, and therefore, there is no cancellations in the $r$-integral. Luckily, the situation when $|x_{\pm}(n/l,y)|\ll T^{\epsilon-2}$ is sufficiently rare so that we can detect all such cases by dividing the sums over $n$ and $l$ into many different ranges. Analyzing \eqref{eq:doublesumint} carefully in these ranges, we complete the proof of Theorem \ref{thm:2moment}.

The paper is organized as follows. All required preliminary results and notation are collected in Section \ref{notation}. The explicit formula for the first twisted moment of symmetric square $L$-functions over Gaussian integers is given in Section \ref{sec:explicit formula}. Section \ref{special} is devoted to the generalization of the results of Farid Khwaja and Olde Daalhuis concerning a uniform asymptotic formula for the Gauss hypergeometric function. Finally, Theorem \ref{thm:2moment} is proved in Section \ref{proof} and Theorem \ref{thm:main}  in Section \ref{PGT}.


\section{Notation and preliminary results}\label{notation}
Let $\kk=\Q(i)$ be the Gaussian number field. All sums in this paper are over Gaussian integers unless otherwise indicated.
For $\Re(s)>1$, the Dedekind zeta function is defined as $$\zeta_{\kk}(s)=4^{-1}\sum_{n\neq 0}|n|^{-2s}.$$
Let $\sigma_{\alpha}(n)=4^{-1}\sum_{d|n}|d|^{2\alpha}$. For $\Re(s)>1$ and $r\in \R$ we have
\begin{equation}\label{sigma series}
\frac{1}{4}\sum_{n\neq0}\frac{\sigma_{ir}(n^2)}{|n|^{2s+2ir}}=\frac{\zeta_{\kk}(s)\zeta_{\kk}(s+ir)\zeta_{\kk}(s-ir)}{\zeta_{\kk}(2s)}.
\end{equation}

Let  $[n,x]=\Re(n\bar{x})$ and
$e[x]=exp(2\pi i\Re(x))$.
For $m,n,c\in \Z[i]$ with $c\neq 0$ the Kloosterman sum is defined by
\begin{equation*}
S(m,n;c)=\sum_{\substack{a\pmod{c}\\ (a,c)=1}}e\left[m\frac{a}{c}+n\frac{a^*}{c}\right], \quad aa^*\equiv 1\pmod{c}.
\end{equation*}
For $m\in\Z$, $\xi\in\C$ and $\Re(s)>1$ let
\begin{equation}\label{Lerch zeta}
\zeta_{\kk}(s;m,\xi)=\sum_{n+\xi\neq0}\left(\frac{n+\xi}{|n+\xi|}\right)^{m}\frac{1}{|n+\xi|^{2s}}.
\end{equation}

\begin{lem}\label{Lerch lemma}
If $m\neq0$ then $\zeta_{\kk}(s;m,\xi)$ is entire function of variable $s$.  Otherwise, it is regular in $s$ except for a simple pole at $s=1$ with residue $\pi.$ For $Re(s)<0$ we have
\begin{equation}\label{LerchFE}
\zeta_{\kk}(s;m,\xi)=(-i)^{|m|}\pi^{2s-1}\frac{\Gamma(1-s+|m|/2)}{\Gamma(s+|m|/2)}
\sum_{n\neq0}\left(\frac{n}{|n|}\right)^{-m}\frac{e[n\xi]}{|n|^{2(1-s)}}.
\end{equation}
\end{lem}
\begin{proof}
Equation \eqref{LerchFE} follows from \cite[Lemma 2]{Mot2001} by making the change of variable $n\rightarrow\bar{n}$ on the right-hand side of \cite[(4.1)]{Mot2001} and using the fact that $e[\bar{n}\bar{\xi}]=e[n\xi]$.
\end{proof}
Let $\Gamma(z)$ be the Gamma function. By Stirling's formula we have
\begin{multline}\label{Stirling2}
\Gamma(\sigma+it)=\sqrt{2\pi}|t|^{\sigma-1/2}\exp(-\pi|t|/2)
\exp\left(i\left(t\log|t|-t+\frac{\pi t(\sigma-1/2)}{2|t|}\right)\right)\\\times
\left(1+O(|t|^{-1})\right)
\end{multline}
for $|t|\rightarrow\infty$ and any  fixed $\sigma$.
Evaluating sufficiently many terms in the asymptotic expansion, it is possible to replace  $O(|t|^{-1})$ in \eqref{Stirling2} by an arbitrarily accurate approximation. 

Let $\{\kappa_j=1+r_{j}^{2}$, $j=1,2,\ldots\}$ be the non-trivial discrete spectrum of the hyperbolic Laplacian on $L^2(\Gamma  \setminus \HH^3)$, and  $\{u_j\}$ be the orthonormal basis of the space of Maaß cusp forms consisting of common eigenfunctions of all Hecke operators and the hyperbolic Laplacian. Each function $u_j$ has the following Fourier expansion
\begin{equation}
u_j(z)=y\sum_{n\neq0}\rho_j(n)K_{ir_j}(2\pi|n|y)e[nx],
\end{equation}
where $K_{\nu}(z)$ is the K-Bessel function of order $\nu.$

The corresponding Rankin-Selberg $L$-function is defined as
\begin{equation}\label{RankinSelberg def}
L(u_j\otimes u_j,s)=\sum_{n\neq0}\frac{|\rho_j(n)|^2}{|n|^{2s}},\quad \Re(s)>1.
\end{equation}
Using the relation to the Hecke eigenvalues $\rho_j(n)=\rho_j(1)\lambda_j(n)$, we obtain
\begin{equation}\label{RankinSelberg tosymsq}
L(u_j\otimes u_j,s)=|\rho_j(1)|^2\frac{\zeta_{\kk}(s)}{\zeta_{\kk}(2s)}L(\sym^2 u_{j},s),
\end{equation}
where
\begin{equation}\label{symsq def}
L(\sym^2 u_{j},s)=\zeta_{\kk}(2s)\sum_{n\neq0}\frac{\lambda_{j}(n^2)}{|n|^{2s}}.
\end{equation}
The symmetric square $L$-function is an entire function (see \cite{shimura_1975}) which satisfies the functional equation
\begin{equation}\label{eq:symfunc}
L(\sym^2 u_{j},s)\gamma(s,r_{j})=L(\sym^2 u_{j},1-s)\gamma(1-s,r_{j}),
\end{equation}
where
\begin{equation}\label{eq:symfunc Gamma}
\gamma(s, r_{j})=\pi^{-3s}\Gamma(s)\Gamma(s+ir_{j})\Gamma(s-ir_{j}).
\end{equation}

The standard normalizing coefficient is defined as follows
\begin{equation}\label{alphaj}
\alpha_{j}=\frac{r_j|\rho_j(1)|^2}{\sinh(\pi r_j)}.
\end{equation}

\begin{lem}
The following approximate functional equation holds
\begin{equation}\label{approx.func.eq.}
L(\sym^2 u_{j},s)=\sum_{l\neq0}\frac{\lambda_j(l^2)}{|l|^{2s}}V(|l|,r_j,s)+
\sum_{l\neq0}\frac{\lambda_j(l^2)}{|l|^{2-2s}}V(|l|,r_j,1-s),
\end{equation}
where for any $y>0$ and $a>0$
\begin{equation}\label{approx.fun.eq.Vdef}
V(y,r_j,s)=\frac{1}{2\pi i}\int_{(a)}\frac{\gamma(s+z,r_j)}{\gamma(s,r_j)}\zeta_{\kk}(2s+2z)\mathfrak{F}(z)y^{-2z}\frac{dz}{z},
\end{equation}
\begin{equation}\label{Gdef}
\mathfrak{F}(z)=\exp(z^2)P_n(z^2)
\end{equation}
with  $P_n$ being a polynomial of degree $n$ such that $P_n(0)=1$.
\end{lem}
\begin{proof}
The proof is similar to \cite[Lemma 7.2.1]{Ng}.
\end{proof}
Using \eqref{eq:symfunc Gamma} and \eqref{Stirling2} we prove the following estimates (see \cite[Lemma 7.2.2]{Ng}).
\begin{lem}
Let $r_j\sim T$ and $s=1/2+it$ with $|t|\ll T^{\epsilon}$. For any positive numbers $y$ and $A$ we have
\begin{equation}\label{Vestimate}
V(y,r_j,s)\ll\left(\frac{r_j^2|t|}{y^2}\right)^{A}.
\end{equation}
For any positive integer $N$ and $1\le y\ll r_j^{1+\epsilon}$
\begin{multline}\label{Vapproximation}
V(y,r_j,s)=\frac{1}{2\pi i}\int_{(a)}
\left(\frac{r_j}{\pi^{3/2}y}\right)^{2z}
\frac{\Gamma(s+z)}{\Gamma(s)}\zeta_{\kk}(2s+2z)\mathfrak{F}(z)\\\times
\left(1+\sum_{k=1}^{N-1}\frac{p_{2k}(v+t)}{r_j^k}\right)\frac{dz}{z}+O(r_j^{-N+\epsilon}),
\end{multline}
where $v=\Im(z)$ and $p_n(v)$ is a polynomial of degree $n$.
\end{lem}


Let
\begin{equation}\label{L beatiful def}
\mathscr{L}_{\kk}(s;n)=\frac{\zeta_{\kk}(2s)}{\zeta_{\kk}(s)}\sum_{q\neq0}\frac{\rho_q(n)}{|q|^{2s}},\quad \Re(s)>1,
\end{equation}
\begin{equation}\label{rho def}
\rho_q(n):=\#\{x\Mod{2q}:x^2\equiv n\Mod{4q}\}.
\end{equation}
We have (see \cite{Szmidt} and \cite[(3.20)]{BF2})
\begin{equation}\label{L beautiful 0}
\mathscr{L}_{\kk}(s;0)=4\zeta_{\kk}(2s-1).
\end{equation}
Furthermore, the following subconvexity bound holds (see \cite[Eq. 3.12]{BF2})
\begin{equation}\label{Lbeaut subconvex}
\mathscr{L}_{\kk}(1/2+it;n)\ll (1+|t|)^A|n|^{2\theta+\epsilon},
\end{equation}
where $ \theta$ can be taken as $1/6$ according to the result of Nelson \cite{N}.
As an analogue of \cite[Lemma 4.1]{BF3}, we derive the following equation relating sums of Kloosterman sums and the $L$-series \eqref{L beatiful def}
\begin{equation}\label{eq:sumofklsums}
\sum_{q\neq0}\frac{1}{|q|^{2+2s}}\sum_{c\Mod{q}}S(l^2,c^2;q)e\left[n\frac{c}{q}\right]=
\frac{\mathscr{L}_{\kk}(s;n^2-4l^2)}{\zeta_{\kk}(2s)}.
\end{equation}

\section{Explicit formula for the first moment}\label{sec:explicit formula}
In this section we evaluate the first twisted moment
\begin{equation}\label{M1 def}
\M_1(l,s;h):=\sum_{j}h(r_j)\alpha_{j}\lambda_{j}(l^2)L(\sym^2 u_{j},s),
\end{equation}
where $h(t)$  is an even function, holomorphic in any fixed horizontal strip and satisfying the conditions
\begin{equation}\label{conditions on h 1}
h(\pm(n-1/2)i)=0,\quad h(\pm ni)=0\quad\hbox{for}\quad n=1,2,\ldots ,N,
\end{equation}
\begin{equation}\label{conditions on h 2}
h(r)\ll\exp(-c|r|^2)
\end{equation}
for some fixed $N$ and $c>0.$

\begin{thm}\label{thm 1moment exact} For $1/2\le\Re(s)<1$ and any even function $h(t)$  that is holomorphic in any fixed horizontal strip and satisfies the conditions \eqref{conditions on h 1} and \eqref{conditions on h 2}, we have
\begin{equation}\label{eq:M1(l,s,h)exact}
\M_1(l,s;h)=MT(l,h;s)+CT(l,h;s)+ET(l,h;s)+\Sigma(l,h;s),
\end{equation}
where
\begin{equation}\label{M1 MT}
MT(l,h;s)=\frac{4\zeta_{\kk}(2s)}{\pi^2|l|^{2s}}\int_{-\infty}^{\infty}r^2h(r)\tanh(\pi r)dr,
\end{equation}
\begin{equation}\label{M1 CT}
CT(l,h;s)=-8\pi\zeta_{\kk}(s)\int_{-\infty}^{\infty}h(r)\frac{\sigma_{ir}(l^2)}{|l|^{2ir}}
\frac{\zeta_{\kk}(s+ir)\zeta_{\kk}(s-ir)}{\zeta_{\kk}(1+ir)\zeta_{\kk}(1-ir)}dr,
\end{equation}
\begin{equation}\label{M1 ET}
ET(l,h;s)=-16\pi^2h(i(s-1))\frac{\zeta_{\kk}(2s-1)}{\zeta_{\kk}(2-s)}
\left(\frac{\sigma_{1-s}(l^2)}{|l|^{2-2s}}+\frac{\sigma_{s-1}(l^2)}{|l|^{2s-2}}\right),
\end{equation}
\begin{equation}\label{M1 Sigma decomposition}
\Sigma(l,h;s)=\Sigma_0(l,h;s)+\Sigma_2(l,h;s)+\Sigma_{gen}(l,h;s),
\end{equation}
\begin{equation}\label{Sigma0 def}
\Sigma_{0}(l,h;s)=\frac{8(2\pi)^{2s-1}}{\pi^2|l|^{2-2s}}
\mathscr{L}_{\kk}(s;-4l^2)\int_0^{\pi/2}I\left(0,\tau,s\right)d\tau,
\end{equation}
\begin{equation}\label{Sigma2 def}
\Sigma_{2}(l,h;s)=\frac{32(2\pi)^{2s-1}\zeta_{\kk}(2s-1)}{\pi^2|l|^{2-2s}}\int_0^{\pi/2}
\sum_{\pm}I\left(\pm2,\tau,s\right)d\tau,
\end{equation}
\begin{equation}\label{Sigma gen def}
\Sigma_{gen}(l,h;s)=\frac{8(2\pi)^{2s-1}}{\pi^2|l|^{2-2s}}\sum_{n\neq0,\pm2l}
\mathscr{L}_{\kk}(s;n^2-4l^2)\int_0^{\pi/2}I\left(\frac{n}{l},\tau,s\right)d\tau,
\end{equation}
\begin{multline}\label{I integral hypergeom1}
I(z,\tau,s)=
\frac{1}{4}\int_{-\infty}^{\infty}r^2h(r)\cosh(\pi r)\sum_{\pm}
\frac{\Gamma(1-s-ir)\Gamma(1-s+ir)}{4(\cos\tau)^{2-2s}}\\ \times
F\left(1-s+ir,1-s-ir,1;-x_{\pm}(z,\tau)\right)dr,
\end{multline}
where for $\vartheta=\arg(z)$ we define
\begin{equation}\label{x def}
x_{\pm}(z,\tau)=\frac{|z|^2+4\sin^2\tau\pm4|z|\sin\tau\cos\vartheta}{(2\cos\tau)^{2}}.
\end{equation}
\end{thm}
\begin{proof}
This result is a generalization of \cite[Theorem 4.13]{BF2} and can be proved in the same way.  We will indicate the required changes in the proof of \cite[Theorem 4.13]{BF2}. 

We start  by assuming that $\Re(s)>3/2$.
Substituting \eqref{symsq def} to \eqref{M1 def}, applying the Kuznetsov formula (see \cite[Theorem 3.2]{BF2}) and using \eqref{sigma series}, we obtain an analogue of \cite[Lemma 4.2]{BF2}, namely
\begin{equation}\label{eq:M1(l,s,h)exact1}
\M_1(l,s;h)=MT(l,h;s)+CT(l,h;s)+\Sigma(s)
\end{equation}
with
\begin{equation}\label{sigmadef}
\Sigma(s)=\zeta_{\kk}(2s)
\sum_{n\neq0}\frac{1}{|n|^{2s}}\sum_{q\neq0}\frac{S(l^2,n^2;q)}{|q|^2}\check{h}\left(\frac{2\pi ln}{q}\right),
\end{equation}
where $\check{h}(z)$ is defined in \cite[(3.15)]{BF2}. Note that the main difference between \eqref{sigmadef} and \cite[(4.4)]{BF2} is in the arguments of Kloosterman sums. Furthermore,  \eqref{sigmadef} contains the additional multiple $\zeta_{\kk}(2s)$. 

The next step is to adapt the proof of \cite[Lemma 4.6]{BF2} to our case. The main changes occur while proving an analogue of \cite[(4.44)]{BF2}. Using  \eqref{LerchFE} and \eqref{eq:sumofklsums} we obtain
\begin{multline}\label{Sigma5}
\sum_{q\neq0}\frac{1}{|q|^{2+2s}}
\sum_{c\Mod{q}}S(l^2,c^2;q)
\zeta_{\kk}(s+w/2;2m,c/q)=\\
(-i)^{|2m|}\pi^{2s+w-1}\frac{\Gamma(1-s-w/2+|m|)}{\Gamma(s+w/2+|m|)}\\\times
\sum_{n\neq0}\left(\frac{n}{|n|}\right)^{-2m}|n|^{2s+w-2}
\sum_{q\neq0}\frac{1}{|q|^{2+2s}}
\sum_{c\Mod{q}}S(l^2,c^2;q)e[nc/q]=\\
(-1)^{|m|}\pi^{2s+w-1}\frac{\Gamma(1-s-w/2+|m|)}{\Gamma(s+w/2+|m|)}
\sum_{n\neq0}\left(\frac{n}{|n|}\right)^{-2m}|n|^{2s+w-2}
\frac{\mathscr{L}_{\kk}(s;n^2-4l^2)}{\zeta_{\kk}(2s)}.
\end{multline}
Applying \eqref{Sigma5} in place of \cite[(4.44)]{BF2}, we finally derive the following analogue of \cite[(4.40)]{BF2}
\begin{multline}\label{Sigma1}
\Sigma(s)=
\frac{8(2\pi)^{2s-1}}{\pi^2|l|^{2-2s}}
\int_0^{\pi/2}\mathscr{L}_{\kk}(s;-4l^2)I(0,\tau,s)d\tau+\\
\frac{8(2\pi)^{2s-1}}{\pi^2|l|^{2-2s}}
\int_0^{\pi/2}\sum_{n\neq0}\mathscr{L}_{\kk}(s;n^2-4l^2)I(n/l,\tau,s)d\tau,
\end{multline}
where $I(0,\tau,s)$ and $I(z,\tau,s)$ are defined by \cite[(4.39)]{BF2} and \cite[(4.38)]{BF2}.  The representation \eqref{I integral hypergeom1} is proved in \cite[Lemma 4.10]{BF2}. Using \eqref{L beautiful 0} we show that for $\Re(s)>1/2$
\begin{equation*}
\Sigma(s)=\Sigma_0(l,h;s)+\Sigma_2(l,h;s)+\Sigma_{gen}(l,h;s).
\end{equation*}

The final step is to continue analytically the term $CT(l,h;s)$ to the region $\Re(s)<1$. Doing so, we obtain the additional summand $ET(l,h;s)$ defined by \eqref{M1 ET} which comes from the poles of $\zeta_{\kk}(s+ir)\zeta_{\kk}(s-ir).$ Thus the theorem is proved for $1/2<\Re(s)<1$. In order to extend our result to the critical line $\Re(s)=1/2$ and specially to the point $s=1/2$, we proceed  in the same way as in \cite[Theorem 4.13, Remark 4.14]{BF2}.
\end{proof}

\section{Special functions}\label{special}
According to  \eqref{I integral hypergeom1}, in order to estimate $I(z,\tau,s)$ it is required to investigate the asymptotic behavior of the Gauss hypergeometric function
\begin{equation}\label{hypergeom1}
F\left(1/2-it+ir,1/2-it-ir,1;-x\right)
\end{equation}
for $r\sim T$, $|t|\ll T^{\epsilon}$ as $T\rightarrow+\infty$ and uniformly for $0<x<\infty.$ When $t=0$, the asymptotic formula for \eqref{hypergeom1} was proved by Jones \cite{Jones} using the Liouville-Green method. Furthermore, this result was reproved by Farid Khwaja and  Olde Daalhuis \cite[Theorem 3.1]{KD2} by means of the saddle point method. 

For our application, it is required to have an asymptotic formula  for \eqref{hypergeom1} for any small $t$. This can be achieved by generalizing \cite[Theorem 3.1]{KD2}.
Accordingly, we introduce two new variables $\lambda=ir$ and $\alpha=t/r$ and consider
\begin{equation}\label{hypergeom2}
F\left(1/2+\lambda(1-\alpha),1/2-\lambda(1+\alpha),1;-x\right)
\end{equation}
for $\lambda\rightarrow\infty$ in $|\arg(\lambda)|\le\pi/2$ and $|\alpha|\ll|\lambda|^{-1+\epsilon}$. Consequently, we derive the following result.

\begin{thm}\label{thm 2F1 asympt}
For $0<x<\infty$ and and $|\alpha|\ll|\lambda|^{-1+\epsilon}$  we have
\begin{multline}\label{Fasymptotic}
F\left( \frac{1}{2}+\lambda(1-\alpha),\frac{1}{2}-\lambda(1+\alpha),1; -x\right)=\\-e^{\lambda\eta}
\Biggl(
I_{0}(\lambda\xi)
\sum_{j=0}^{n-1}\frac{a_j}{\lambda^j}+
\frac{2}{\xi}I_{1}(\lambda\xi)
\sum_{j=1}^{n-1}\frac{b_j}{\lambda^j}+O(\Phi_n(\lambda,\xi))
\Biggr)
\end{multline}
as $\lambda\rightarrow\infty$ in $|\arg(\lambda)|\le\pi/2$, where
\begin{equation}\label{def:eta}
\eta=\alpha\log\left(1+x\right),
\end{equation}
\begin{multline}\label{def:xi}
\xi=\alpha\log\left(1+x\right)-(1+\alpha)\log\left(\frac{\alpha x+\sqrt{x^2+x(1-\alpha^2)}}{x+\sqrt{x^2+x(1-\alpha^2)}}\right)+\\
(1-\alpha)\log\left(\frac{1-\alpha+x+\sqrt{x^2+x(1-\alpha^2)}}{1-\alpha}\right),
\end{multline}
\begin{equation}\label{def:a0}
a_0=-\frac{\xi^{1/2}}{2^{1/2}(x^2+x(1-\alpha^2))^{1/4}},
\end{equation}
\begin{equation}
\Phi_n(\lambda,\xi)=\frac{1}{|\lambda|^n} \biggl( |I_{0}(\lambda\xi)|+\frac{|I_{1}(\lambda\xi)|}{|\xi |}\biggr).
\end{equation}
\end{thm}
\begin{proof}
Using  \cite[(3.14)]{KD2} with $a=1/2-\lambda\alpha$, $z=1+2x$ we obtain
\begin{equation}\label{F integral}
F\left( \frac{1}{2}+\lambda(1-\alpha),\frac{1}{2}-\lambda(1+\alpha),1; -x\right)=
\frac{1}{2\pi i}\int_{-\infty}^{(0+)}g(t)e^{\lambda f(t)}dt,
\end{equation}
where
\begin{equation}\label{f and g def}
f(t)=(1+\alpha)\log\left(\frac{(1+x)-t}{1-t}\right)-(1-\alpha)\log(t),\quad
g(t)=\frac{1}{t^{1/2}(1-t)}\left(\frac{1-t}{1+x-t}\right)^{1/2}.
\end{equation}
As in \cite{KD2} the branch-cuts are $(-\infty,0)$ and $(1,1+x).$  Since
\begin{equation*}
f'(t)=(1+\alpha)\left(\frac{1}{t-1-x}-\frac{1}{t-1}\right)-\frac{1-\alpha}{t},
\end{equation*}
the saddle points (the solutions of $f'(t)=0$) are
\begin{equation}\label{sp pm}
\sap_{\pm}=\frac{x+1-\alpha\pm\sqrt{x^2+x(1-\alpha^2)}}{1-\alpha}.
\end{equation}
Similarly to \cite[(3.18)]{KD2}, we make the change of variable
\begin{equation}\label{variable change}
f(t)=\tau+\frac{\xi^2}{4\tau}+\eta.
\end{equation}
If $x=\alpha=0$ this transformation reduces to $-\log t=\tau$, and  therefore, the point $t=1$ corresponds to $\tau=0$. Thus under the transformation \eqref{variable change} we have that as $\tau$ grows from $-\infty$ to $0$, the variable $t$ decays from $+\infty$ to $1+x$, and $\tau=-\xi/2$ corresponds to $t=\sap_{+}$. As $\tau$ grows from $0$ to $+\infty$, the  variable $t$ decays from $1$ to $0$, and $\tau=\xi/2$ corresponds to $t=\sap_{-}.$ An important consequence of these observations is that
$\frac{dt}{d\tau}<0.$ Solving the following system
$$\left\{
  \begin{array}{ll}
    f(\sap_{+})=-\xi+\eta  \\
    f(\sap_{-})=\xi+\eta
  \end{array}
\right. ,
$$
we obtain the equality \eqref{def:xi} for $\xi$ and \eqref{def:eta} for $\eta$. Finally, we prove an analogue of \cite[(3.20), (3.25)]{KD2}, namely
\begin{equation}\label{F integral2}
F\left( \frac{1}{2}+\lambda(1-\alpha),\frac{1}{2}-\lambda(1+\alpha),1; -x\right)=-
\frac{e^{\lambda\eta}}{2\pi i}\int_{\mathfrak{C}}G_0(\tau)e^{\lambda(\tau+\frac{\xi^2}{4\tau})}\frac{d\tau}{\tau},
\end{equation}
where (see \cite[(3.21)]{KD2})
\begin{equation}\label{G0 def}
G_0(\tau)=g(t)\frac{dt}{d\tau}\tau
\end{equation}
and $\mathfrak{C}$ is the steepest descent contour (see \cite[Figure 3]{KD2})\footnote{Note that there is a typo in  \cite[(3.20),(3.25)]{KD2}: the minus sign is missed. This is because the contour $\mathfrak{C}$ is taken in the opposite direction to the contour obtained after the change of variable \eqref{variable change}.}.
Following \cite{KD2} we obtain \eqref{Fasymptotic}. It is left to evaluate $a_0$, $b_0$ which are defined by (see \cite[(3.31)]{KD2})
\begin{equation}\label{a0b0 def}
a_0=\frac{1}{2}\left(G_0(\xi/2)+G_0(-\xi/2)\right),\quad
b_0=\frac{\xi}{4}\left(G_0(\xi/2)-G_0(-\xi/2)\right).
\end{equation}
As a consequence of \eqref{G0 def}  we find that
\begin{equation}\label{G0 pmxi/2}
G_0(\pm\xi/2)=\pm\frac{\xi}{2}g(\sap_{\mp})\frac{dt}{d\tau}\Biggl|_{\pm\xi/2}.
\end{equation}
Next, we evaluate $g(\sap_{\mp}).$ Since $0<\sap_{-}<1$ we have
\begin{equation}\label{g sp minus}
g(\sap_{-})=\frac{1}{\sqrt{\sap_{-}(1-\sap_{-})(1+x-\sap_{-})}}=\left(\frac{1-\alpha}{1+\alpha}\right)^{1/2}\frac{1}{x^{1/2}\cdot\sap_{-}}.
\end{equation}
Since $\sap_{+}>1+x$ the following holds
\begin{equation}\label{g sp plus}
g(\sap_{+})=\frac{-1}{\sqrt{\sap_{+}(\sap_{+}-1)(\sap_{+}-1-x)}}=-\left(\frac{1-\alpha}{1+\alpha}\right)^{1/2}\frac{1}{x^{1/2}\cdot\sap_{+}}.
\end{equation}
Differentiating \eqref{variable change} we obtain $f'(t)\frac{dt}{d\tau}=1-\frac{\xi^2}{4\tau^2}.$ Taking one more derivative we infer
\begin{equation*}
f''(t)\left(\frac{dt}{d\tau}\right)^2+f'(t)\frac{d^2t}{d\tau^2}=\frac{\xi^2}{2\tau^3}\quad\Longrightarrow\quad
\left(\frac{dt}{d\tau}\right)^2\Biggl|_{\pm\xi/2}=\frac{\pm4}{\xi f''(\sap_{\mp})}.
\end{equation*}
Evaluating the second derivative at the saddle points
\begin{equation*}
f''(\sap_{\pm})=\mp2\frac{1-\alpha}{1+\alpha}\cdot\frac{\sqrt{x^2+x(1-\alpha^2)}}{x\cdot \sap^2_{\pm}}
\end{equation*}
and using the fact that $\frac{dt}{d\tau}<0$, we have
\begin{equation}\label{dt/dtau}
\frac{dt}{d\tau}\Biggl|_{\pm\xi/2}=-\left(\frac{1+\alpha}{1-\alpha}\right)^{1/2}
\frac{(2x)^{1/2}\cdot \sap_{\mp}}{\xi^{1/2}(x^2+x(1-\alpha^2))^{1/4}}.
\end{equation}
Substituting \eqref{g sp minus}, \eqref{g sp plus} and \eqref{dt/dtau} to \eqref{G0 pmxi/2}, we conclude that
\begin{equation}\label{G0 pmxi/2 2}
G_0(\pm\xi/2)=-\frac{(\xi/2)^{1/2}}{(x^2+x(1-\alpha^2))^{1/4}}.
\end{equation}
Substituting \eqref{G0 pmxi/2 2} to \eqref{a0b0 def}, we prove \eqref{def:a0} and show that $b_0=0$.
\end{proof}
We will apply Theorem \ref{thm 2F1 asympt} with $\lambda=ir$. In this case $I$--Bessel functions can be written in terms of $J$--Bessel functions (see \cite[10.27.6]{HMF}) 
\begin{equation}\label{ItoJ}
I_0(ir\xi)=J_0(r\xi),\quad I_1(ir\xi)=iJ_1(r\xi).
\end{equation}
Furthermore, it follows from \eqref{def:xi} that as $x\rightarrow0$ and $x\rightarrow\infty$ we respectively have
\begin{equation}\label{xi x 0}
\xi=2\sqrt{x(1-\alpha^2)}+O(x^{3/2}),\quad
\xi=\log(x)+O(1).
\end{equation}
Also it will be required to study $\xi$ as a function of $\alpha$. In this case,  as $\alpha\rightarrow0$ we have
\begin{equation}\label{xi alpha 0}
\xi(\alpha)=\xi(0)+O(\alpha^2),\quad\xi(0)=\log(1+2x+2\sqrt{x^2+x}).
\end{equation}

\section{Proof of Theorem \ref{thm:2moment}}\label{proof}
Following the paper of Ivic and Jutila \cite{IvJut}, let us define
\begin{equation}\label{omega def}
\omega_T(r)=\frac{1}{G\pi^{1/2}}\int_T^{2T}\exp\left(-\frac{(r-K)^2}{G^2}\right)dK.
\end{equation}
For an arbitrary $A>1$ and some $c>0$ we have (see \cite{IvJut})
\begin{equation}\label{omega1}
\omega_T(r)=1+O(r^{-A})\text{ if } T+cG\sqrt{\log T}<r<2T-cG\sqrt{\log T},
\end{equation}
\begin{equation}\label{omega2}
\omega_T(r)=O((|r|+T)^{-A})\text{ if }
r<T-cG\sqrt{\log T}\text{ or } r>2T+cG\sqrt{\log T},
\end{equation}
and otherwise
\begin{equation}\label{omega3}
\omega_T(r)=1+O(G^3(G+\min(|r-T|,|r-2T|))^{-3}).
\end{equation}
To prove Theorem \ref{thm:2moment} we consider
\begin{equation*}
\M_2(T):=\sum_{r_j}\alpha_{j}\omega_T(r_j)L(\sym^2 u_{j},1/2+it)L(\sym^2 u_{j},1/2-it),
\end{equation*}
where $\omega_T(r)$ is defined by \eqref{omega def} with $G=T^{1-\epsilon}$ and $|t|\ll T^{\epsilon}.$  Applying  the approximate functional equation \eqref{approx.func.eq.}  for $L(\sym^2 u_{j},1/2-it)$ and using \eqref{Vestimate}, we infer
\begin{equation*}
\M_2(T)\ll \sum_{|l|\ll T^{1+\epsilon}}\frac{1}{|l|}
\left|\sum_{r_j}\alpha_{j}\omega_T(r_j)V(|l|,r_j,1/2\pm it)\lambda_j(l^2)L(\sym^2 u_{j},1/2+it)\right|.
\end{equation*}
Then it follows from \eqref{omega def} that
\begin{equation}\label{M2 estimate0}
\M_2(T)\ll \sum_{|l|\ll T^{1+\epsilon}}\frac{1}{|l|}
\left|\frac{1}{G\pi^{1/2}}\int_T^{2T}\M_1\left(l,s,h(\cdot)V(|l|,\cdot,1/2\pm it)\right)dK\right|,
\end{equation}
where $s=1/2+it$ and as in \eqref{M1 def} we have
\begin{equation}\label{M12 def}
\M_1\left(l,s;h(\cdot)V(|l|,\cdot,1/2\pm it)\right):=\sum_{j}h(r_j)V(|l|,r_j,1/2\pm it)
\alpha_{j}\lambda_{j}(l^2)L(\sym^2 u_{j},s)
\end{equation}
with
\begin{equation}\label{hN def}
h(r)=q_N(r)\exp\left(-\frac{(r-K)^2}{G^2}\right)+q_N(r)\exp\left(-\frac{(r+K)^2}{G^2}\right),
\end{equation}
\begin{equation}\label{qN def}
q_N(r)=\frac{(r^2+1/4)\ldots(r^2+(N-1/2)^2)}{(r^2+100N^2)^N}
\frac{(r^2+1)\ldots(r^2+N^2)}{(r^2+100N^2)^N}.
\end{equation}

To evaluate \eqref{M12 def} we apply Theorem \ref{thm 1moment exact} and then substitute the result in \eqref{M2 estimate0}. This way the contribution of \eqref{M1 MT}, \eqref{M1 CT} and \eqref{M1 ET} is bounded by $O(T^{3+\epsilon}).$ Thus to prove Theorem \ref{thm:2moment} it is required to show that
\begin{equation}\label{M2 estimate1}
\M_2(T)\ll\sum_{|l|\ll T^{1+\epsilon}}\frac{1}{|l|}
\left|\frac{1}{G\pi^{1/2}}\int_T^{2T}\Sigma\left(l,s,h(\cdot)V(|l|,\cdot,1/2\pm it)\right)dK\right|\ll T^{3+4\theta+\epsilon},
\end{equation}
where $\Sigma(l,h;s)$ is defined by \eqref{M1 Sigma decomposition}. 

The most difficult part is to estimate the term which involves the $\Sigma_{gen}(l,h;s)$ summand of  $\Sigma(l,h;s)$. The contribution of the term with $\Sigma_0(l,h;s)$ can be estimated similarly (see also \cite[(6.42), (6.43)]{BF2}). The summand  $\Sigma_2(l,h;s)$ is a part of the main term (see \cite[p.24]{BF2}), and therefore, is of size $O(T^{3+\epsilon}).$ Furthermore,  it is sufficient to consider only one ("+" or "-" ) case for $V(|l|,\cdot,1/2\pm it)$ since both cases can be treated in the same way. 
Consequently, using \eqref{Lbeaut subconvex} and \eqref{Sigma gen def} and making the change of variable $y=\sin\tau$, we prove that
\begin{equation}\label{M2 estimate2}
\M_2(T)\ll  T^{3+\epsilon}+\sum_{|l|\ll T^{1+\epsilon}}
\sum_{\pm}\frac{S_{\pm}(l)}{|l|^2},
\end{equation}
where
\begin{equation}\label{S(l) def}
S_{\pm}(l)=
\sum_{n\neq0,2l,-2l}|n^2-4l^2|^{2\theta}
\left|\int_0^{1}\frac{1}{G\pi^{1/2}}\int_T^{2T}
I_{\pm}\left(\frac{n}{l},y,s\right)dK\frac{dy}{(1-y^2)^{1/2}}\right|
\end{equation}
and (see \eqref{I integral hypergeom1})
\begin{multline}\label{I integral hypergeom1V}
I_{\pm}(z,y,s)=
\int_{-\infty}^{\infty}r^2h(r)V(|l|,r,s)\cosh(\pi r)
\frac{\Gamma(1-s-ir)\Gamma(1-s+ir)}{(1-y^2)^{1-s}}\\ \times
F\left(1-s+ir,1-s-ir,1;-x_{\pm}(z,y)\right)dr.
\end{multline}
Here  for $\vartheta=\arg(z)$ we have
\begin{equation}\label{x def2}
x_{\pm}(z,y)=\frac{f_{\pm}(z,y)}{1-y^2},\quad f_{\pm}(z,y)=y^2\pm|z|y\cos\vartheta+|z/2|^2.
\end{equation}
According to \cite[Lemma 4.11]{BF2}
\begin{multline}\label{I integral hypergeom2}
I_{\pm}(z,y,s)=\frac{2}{(x_{\pm}(z,y)(1-y^2))^{1-s}}
\int_{-\infty}^{\infty}r^2h(r)\cosh(\pi r)x_{\pm}^{-ir}(z,y)\\
\times\frac{\Gamma(1-s+ir)\Gamma(-2ir)}{\Gamma(s-ir)}
F\left(1-s+ir,1-s+ir,1+2ir;\frac{-1}{x_{\pm}(z,y)}\right)dr.
\end{multline}

First, we show that the contribution of large $|n|$ in \eqref{M2 estimate2} is negligible.  To this end, we prove an analogue of \cite[Lemma 4.12]{BF2}. The only difference in our case is the presence of the function $V(|l|,r,s)$ in \eqref{I integral hypergeom1V}.
Moving the line of integration in  \eqref{approx.fun.eq.Vdef} to $\Re(z)=M$ we obtain
\begin{equation*}
V(|l|,r-iM,s)\ll\left(\frac{r}{|l|}\right)^{2M}.
\end{equation*}
Using this estimate and following the proof of \cite[Lemma 4.12]{BF2} we infer
\begin{equation*}
\int_0^{1}\frac{1}{G\pi^{1/2}}\int_T^{2T}
I_{\pm}\left(z,\tau,s\right)dK\frac{dy}{(1-y^2)^{1/2}}\ll T^3\left(\frac{T}{|lz|}\right)^{2N+1}.
\end{equation*}
Consequently, the contribution of $n$ such that $|n|\gg T^{\epsilon}|l|$ to \eqref{S(l) def} is negligible.

Next, we consider the contribution of small $n$. Let
\begin{equation}\label{Jdef}
J_{\pm}\left(\frac{n}{l},y,T\right):=\frac{1}{G\pi^{1/2}}\int_T^{2T}I_{\pm}\left(\frac{n}{l},y,s\right)dK.
\end{equation}
Then we have
\begin{equation}\label{S(l) est0}
S_{\pm}(l)\ll
\sum_{\substack{n\neq0,2l,-2l\\|n|\ll T^{\epsilon}|l|}}|n^2-4l^2|^{2\theta}
\left|\int_0^{1}
J_{\pm}\left(\frac{n}{l},y,T\right)\frac{dy}{(1-y^2)^{1/2}}\right|+T^{-A}.
\end{equation}

\begin{lem}\label{lemma I estimate ynear1}
Suppose that for $1-T^{-B}<y<1$ the following inequalities hold:  $f_{\pm}(z,y)>T^{-a}$ and $B-5>a>0$. Then
\begin{equation}\label{I integral estimate ynear1}
\int_{1-T^{-B}}^{1}
J_{\pm}\left(\frac{n}{l},y,T\right)
\frac{dy}{(1-y^2)^{1/2}}
\ll T^{(5+a-B)/2}.
\end{equation}
\end{lem}
\begin{proof}
It follows from \eqref{x def2} and the statement of the lemma that $x_{\pm}(z,y)^{-1}\ll T^{a-B}<T^{-5}.$ Thus the hypergeometric function in \eqref{I integral hypergeom2} can be estimated by a constant. Consequently,
\begin{equation*}
J_{\pm}\left(\frac{n}{l},y,T\right)
\ll\frac{T^{5/2}}{f_{\pm}(z,y)^{1/2}}\ll T^{(5-a)/2}.
\end{equation*}
Then the lemma follows by substituting this estimate in \eqref{I integral estimate ynear1}.
\end{proof}

\begin{lem}\label{lemma I estimate main}
For $0\le y<1$, $|x_{\pm}(z,y)|\gg T^{\epsilon-2}$ and an arbitrary fixed $A>0$ we have
\begin{equation}\label{I integral estimate2}
J_{\pm}\left(\frac{n}{l},y,T\right)
\ll\frac{T^{-A}}{(1-y^2)^{1/2}}.
\end{equation}
For $0\le y<1$, $|x_{\pm}(z,y)|\ll T^{\epsilon-2}$ we have
\begin{equation}\label{I integral estimate3}
J_{\pm}\left(\frac{n}{l},y,T\right)
\ll\frac{T^{3+\epsilon}}{(1-y^2)^{1/2}}.
\end{equation}
\end{lem}
\begin{proof}
To prove the required estimates we substitute \eqref{Fasymptotic} to \eqref{I integral hypergeom1V}. All terms coming from \eqref{Fasymptotic} can be estimated in the same way and the error term is negligible.  Using \eqref{ItoJ} we obtain
\begin{multline}\label{Iest1}
J_{\pm}\left(\frac{n}{l},y,T\right)
=\frac{e^{it\log\left(1+x_{\pm}(z,y)\right)}}{(1-y^2)^{1-s}}
\frac{1}{G\pi^{1/2}}\int_T^{2T}\int_{-\infty}^{\infty}r^2h(r)V(|l|,r,s)\\
\times\cosh(\pi r)\Gamma(1-s-ir)\Gamma(1-s+ir)J_0(r\xi)drdK+\ldots
\end{multline}

Suppose that $|x_{\pm}(z,y)|\ll T^{\epsilon-2}$. Since $r\sim T$ it follows from \eqref{xi x 0} that $r\xi\ll T^{\epsilon}.$  Estimating everything on the right hand side of \eqref{Iest1} trivially we prove \eqref{I integral estimate3}.

Suppose that $|x_{\pm}(z,y)|\gg T^{\epsilon-2}$. In this case we have $r\xi\gg T^{\epsilon}.$ Thus we can apply the asymptotic formula  \cite[8.451.1]{GR} for $J_0(r\xi)$ (again it is enough to consider only the main term). Consequently, it is required to evaluate
\begin{equation*}
\frac{1}{\xi^{1/2}}\int_{-\infty}^{\infty}r^{3/2}h(r)V(|l|,r,s)\cosh(\pi r)\Gamma(1-s-ir)\Gamma(1-s+ir)e^{ir\xi}dr.
\end{equation*}
Using \eqref{hN def} and the Stirling formula \eqref{Stirling2} for the Gamma functions, we infer (see \cite[(6.12)]{BF2})
\begin{multline}\label{Iest3}
J_{\pm}\left(\frac{n}{l},y,T\right)
\ll
\frac{(1-y^2)^{-1/2}}{G\pi^{1/2}}\int_T^{2T}K^{2it}
\int_{-\infty}^{\infty}\frac{r^{3/2}}{\xi^{1/2}}V(|l|,r,s)\exp\left(-\frac{(r-K)^2}{G^2}+ir\xi\right)drdK.
\end{multline}
In the integral over $r$ we first make the change of variable $r=K+Gv$. Note that $\xi$ depends on $\alpha$ (and therefore depends on $r$). To overcome this difficulty, we expand it in Taylor series \eqref{xi alpha 0} at $\alpha=0$ with sufficiently many terms. As a result, we obtain
\begin{multline}\label{Iest4}
J_{\pm}\left(\frac{n}{l},y,T\right)
\ll\frac{\xi(0)^{-1/2}}{(1-y^2)^{1/2}}
\int_T^{2T}K^{3/2+2it}e^{iK\xi(0)}\\\times
\int_{-\infty}^{\infty}V(|l|,K+Gv,s)\exp\left(-v^2+iGv\xi(0)\right)dvdK,
\end{multline}
where
$$\xi(0)=\log(1+2x_{\pm}(z,y)+2\sqrt{x_{\pm}(z,y)^2+x_{\pm}(z,y)}).$$
In order to evaluate the integral over $v$, it is convenient to use the representation \eqref{Vapproximation} for $V(|l|,K+Gv,s)$. Before doing this, at the cost of a negligible error term we truncate the integral over $v$ at $v=\pm(\log T)^2$.  Now we apply \eqref{Vapproximation} (with $a$ being chosen as a small fixed $\epsilon_1>0$). It is sufficient to consider only the main term from \eqref{Vapproximation} since all other terms can be treated in the same way and are smaller in size. The error term coming from the remainder in \eqref{Vapproximation} is negligible. Therefore,  we have
\begin{multline*}
\int_{-\infty}^{\infty}V(|l|,K+Gv,s)\exp\left(-v^2+iGv\xi(0)\right)dv=
\int_{-(\log T)^2}^{(\log T)^2}\exp(-v^2+iGv\xi(0))\\\times
\frac{1}{2\pi i}\int_{(\epsilon_1)}
\left(\frac{K+Gv}{\pi^{3/2}|l|}\right)^{2u}
\frac{\Gamma(s+u)}{\Gamma(s)}\zeta_{\kk}(2s+2u)\mathfrak{F}(u)\frac{du}{u}dv+\ldots
\end{multline*}
Since  the function $\mathfrak{F}(u)$ decays exponentially we can truncate the integral over $u$ at $|\Im{u}|=\log T$ with a negligibly small error term. Consequently, we have $|Gvu/K|\ll T^{-\epsilon}$, and therefore, it is possible to replace  $(K+Gv)^u$ by $K^u$ since
\begin{equation*}
(K+Gv)^u=K^u\left(1+\frac{Gvu}{K}+\ldots\right).
\end{equation*}
As a result,
\begin{multline*}
\int_{-\infty}^{\infty}V(|l|,K+Gv,s)\exp\left(-v^2+iGv\xi(0)\right)dv=
\frac{1}{2\pi i}\int_{(\epsilon_1)}
\left(\frac{K}{\pi^{3/2}|l|}\right)^{2u}
\frac{\Gamma(s+u)}{\Gamma(s)}\\\times
\zeta_{\kk}(2s+2u)\mathfrak{F}(u)
\int_{-(\log T)^2}^{(\log T)^2}\exp(-v^2+iGv\xi(0))dv
\frac{du}{u}+\ldots
\end{multline*}
Now we can extend the integral over $v$ to the whole real line at a cost of a negligible error term. Evaluating the resulting integral, we obtain
\begin{multline}\label{Iest5}
\int_{-\infty}^{\infty}V(|l|,K+Gv,s)\exp\left(-v^2+iGv\xi(0)\right)dv=
\pi^{1/2}\exp\left(-G^2\xi(0)^2\right)\\\times
\frac{1}{2\pi i}\int_{(\epsilon_1)}
\left(\frac{K}{\pi^{3/2}|l|}\right)^{2u}
\frac{\Gamma(s+u)}{\Gamma(s)}\zeta_{\kk}(2s+2u)\mathfrak{F}(u)
\frac{du}{u}+\ldots
\end{multline}
Substituting \eqref{Iest5} to \eqref{Iest4}  and estimating the integrals over $u$ and $K$ trivially, we prove that
\begin{equation*}
J_{\pm}\left(\frac{n}{l},y,T\right)
\ll
\frac{\exp\left(-G^2\xi^2(0)\right)+T^{-A}}{(1-y^2)^{1/2}\xi(0)^{1/2}}T^{5/2+\epsilon}.
\end{equation*}
Finally, since $G=T^{1-\epsilon_2}$ and $|x_{\pm}(z,y)|\gg T^{\epsilon-2}$ we have $G\xi(0)\gg T^{\epsilon_3}$. This completes the proof of \eqref{I integral estimate2}.
 \end{proof}
The next step is to apply Lemmas \ref{lemma I estimate ynear1} and \ref{lemma I estimate main} for estimating \eqref{S(l) def}. To prove Theorem \ref{thm:2moment} it is enough to show (see \eqref{M2 estimate2}) that
\begin{equation}\label{S(l) estimate}
\sum_{|l|\ll T^{1+\epsilon}}\sum_{\pm}\frac{S_{\pm}(l)}{|l|^2}\ll
\sum_{|l|\ll T^{1+\epsilon}}\sum_{\pm}\frac{1}{|l|^2}
\sum_{\substack{n\neq0,2l,-2l,\\|n|\ll T^{\epsilon}|l|}}|n^2-4l^2|^{2\theta}
\left|\int_0^{1}
J_{\pm}\left(\frac{n}{l},y,T\right)\frac{dy}{(1-y^2)^{1/2}}\right|
\ll T^{3+4\theta+\epsilon},
\end{equation}
where for $S_{\pm}(l)$  we used the estimate \eqref{S(l) est0}. We consider further only $S_{-}(l)$ since the sum $S_{+}(l)$ can be estimated similarly. Moreover, without loss of generality we may assume that $l$ belongs to the first quadrant (we  denote this as $l\in I$).
To prove Theorem \ref{thm:2moment} it is enough to show that
\begin{equation}\label{desired est1}
\sum_{\substack{|l|\ll T^{1+\epsilon}\\l\in I}}\frac{S_{-}(l)}{|l|^2}\ll T^{3+4\theta+\epsilon},
\end{equation}
\begin{equation}\label{desired est2}
S_{-}(l)=
\sum_{n\in N(l,T)}
|n^2-4l^2|^{2\theta}\left|\int_0^{1}
J_{-}\left(\frac{n}{l},y,T\right)\frac{dy}{(1-y^2)^{1/2}}\right|,
\end{equation}
where $N(l,T)=\left\{n:  n\neq0,\pm2l,\\|n|\ll T^{\epsilon}|l|\right\}$.

First, we decompose the sum $S_{-}(l)$ into two parts depending on whether  $\cos\vartheta=\cos(\arg(n/l))\le0$ or  $\cos\vartheta=\cos(\arg(n/l))>0,$ namely
\begin{equation}\label{decomposition 1}
S_{-}(l)=S_{-}^{1}(l)+S_{-}^{2}(l),
\end{equation}
\begin{equation}\label{Scpm def}
S_{-}^{1}(l)=
\sum_{\substack{n\in N(l,T)\\\cos\vartheta\leq 0}}
|n^2-4l^2|^{2\theta}\left|\int_0^{1}J_{-}\left(\frac{n}{l},y,T\right)\frac{dy}{(1-y^2)^{1/2}}\right|,
\end{equation}
\begin{equation}\label{Scpm def}
S_{-}^{2}(l)=
\sum_{\substack{n\in N(l,T)\\\cos\vartheta>0}}
|n^2-4l^2|^{2\theta}\left|\int_0^{1}J_{-}\left(\frac{n}{l},y,T\right)\frac{dy}{(1-y^2)^{1/2}}\right|.
\end{equation}

\begin{lem}\label{lemma Sc- est}
The following estimate holds
\begin{equation}\label{Sc- est}
\sum_{|l|\ll T^{1+\epsilon}}\frac{S_{-}^{1}(l)}{|l|^2}\ll T^{2+4\theta+\epsilon}.
\end{equation}
\end{lem}
\begin{proof}
Since  $\cos\vartheta=\cos(\arg(n/l))\le0$ we have (see \eqref{x def2})
\begin{equation}\label{x est0}
x_{-}(n/l,y)\ge\frac{y^2+|n/(2l)|^2}{1-y^2}.
\end{equation}
Let $B$ be a fixed large number. Since $|n/(2l)|\gg T^{-1-\epsilon}$ Lemma \ref{lemma I estimate ynear1} implies that
\begin{equation*}
\int_{1-T^{-B}}^{1}
J_{\pm}\left(\frac{n}{l},y,T\right)\frac{dy}{(1-y^2)^{1/2}}\ll T^{-A},
\end{equation*}
where $A$ can be made arbitrary large by taking sufficiently large $B$.

 Let us now consider the case $y<1-T^{-B}$.
Due to \eqref{I integral estimate2} the contribution of $n$ such that $x_{-}(n/l,y)\gg T^{-2+\epsilon}$ is negligible.
When $x_{-}(n/l,y)\ll T^{-2+\epsilon}$ we apply \eqref{I integral estimate3}. It follows from \eqref{x est0} that in this case $y\ll T^{-1+\epsilon}$ and $|n|\ll|l|T^{-1+\epsilon}\ll T^{\epsilon}$. Therefore, the contribution of such $n$ can be estimated by
\begin{equation}\label{S(l) est1}
S_{-}^{1}(l)\ll\sum_{|n|\ll T^{\epsilon}}|n^2-4l^2|^{2\theta}
\int_0^{T^{-1+\epsilon}}
\frac{T^{3+\epsilon}}{1-y}dy\ll|l|^{4\theta}T^{2+\epsilon}.
\end{equation}
Summing \eqref{S(l) est1} over $|l|\ll T^{1+\epsilon}$ we obtain \eqref{Sc- est}.
\end{proof}
Next, we consider the case $\cos\vartheta=\cos(\arg(n/l))>0.$ Let
\begin{equation}\label{2l,n abcd}
2l=a+ib,\quad n=c+id,
\end{equation}
where $a,b,c,d$ are integers (and $a,b\ge0$). Since $\vartheta=\arg(n/2l)$ we have
\begin{equation}\label{yn def}
y_n:=\Re\left(\frac{n}{2l}\right)=\left|\frac{n}{2l}\right|\cos\vartheta=\frac{ac+bd}{a^2+b^2},
\end{equation}
\begin{equation}\label{sn def}
s_n:=\Im\left(\frac{n}{2l}\right)=\left|\frac{n}{2l}\right|\sin\vartheta=\frac{ad-bc}{a^2+b^2},
\end{equation}
\begin{equation}\label{snyn def}
y^2-2y\left|\frac{n}{2l}\right|\cos\vartheta+\left|\frac{n}{2l}\right|^2=(y-y_n)^2+s_n^2.
\end{equation}
Since $|2l|^2=a^2+b^2\ll T^{2+\epsilon}$  either $y_n=1,\, s_n=1$ or
\begin{equation}\label{ynsn notnear1}
|y_n-1|,|s_n-1|\gg\frac{1}{T^{2+\epsilon}}.
\end{equation}
Using \eqref{yn def}, \eqref{sn def}, \eqref{snyn def} it is possible to rewrite  \eqref{x def2} as
\begin{equation}\label{f-to ynsn}
f_{-}(n/l,y)=(y-y_n)^2+s_n^2,\quad
x_{-}(n/l,y)=\frac{(y-y_n)^2+s_n^2}{1-y^2}.
\end{equation}

It follows from Lemma \ref{lemma I estimate main} that the main contribution to $S_{-}^{2}(l)$ comes from $n$ such that  $x_{-}(n/l,y)\ll T^{-2+\epsilon}.$ This may happen (see \eqref{f-to ynsn}) only if $(y-y_n)^2+s_n^2\ll T^{-2+\epsilon}$.
Furthermore, the contribution of $y$ close to $1$ should be treated separately by Lemma \ref{lemma I estimate ynear1}.
We split the  sum over $n$ in $S_{-}^{2}(l)$ into several parts depending on values of $y_n$ and  $s_n$ as follows
\begin{equation}\label{decomposition 2}
S_{-}^{2}(l)=\sum_{j=1}^{8}S_j(l), \quad 
S_{j}(l)=
\sum_{\substack{n\in N_j(l,T)\\\cos\vartheta>0}}
|n^2-4l^2|^{2\theta}\left|\int_0^{1}J_{-}\left(\frac{n}{l},y,T\right)\frac{dy}{(1-y^2)^{1/2}}\right|,
\end{equation}
where
\begin{equation}\label{N1 def}
N_1(l,T)=\left\{n\in N(l,T)\Bigl| |s_n|\gg T^{-1+\epsilon}\right\},
\end{equation}

\begin{equation}\label{N2 def}
N_2(l,T)=\left\{n\in N(l,T)\Bigl| y_n=1, |s_n|\ll T^{-1+\epsilon}\right\},
\end{equation}

\begin{equation}\label{N3 def}
N_3(l,T)=\left\{n\in N(l,T)\Bigl| y_n\le1,  s_n=0 \right\},
\end{equation}

\begin{equation}\label{N4 def}
N_4(l,T)=\left\{n\in N(l,T)\Bigl| y_n<1-\epsilon_1, 0<|s_n|\ll T^{-1+\epsilon}\right\},
\end{equation}

\begin{equation}\label{N5 def}
N_5(l,T)=\left\{n\in N(l,T)\Bigl| 1-T^{-2+\epsilon}<y_n<1, 0<|s_n|\ll T^{-1+\epsilon}\right\},
\end{equation}

\begin{equation}\label{N6 def}
N_6(l,T)=\left\{n\in N(l,T)\Bigl| 1-\epsilon_1<y_n<1-T^{-2+\epsilon}, 0<|s_n|\ll T^{-1+\epsilon}\right\},
\end{equation}

\begin{equation}\label{N7 def}
N_7(l,T)=\left\{n\in N(l,T)\Bigl| y_n>1+T^{-2+\epsilon}, 0<|s_n|\ll T^{-1+\epsilon}\right\},
\end{equation}

\begin{equation}\label{N8 def}
N_8(l,T)=\left\{n\in N(l,T)\Bigl| 1<y_n<1+T^{-2+\epsilon}, 0<|s_n|\ll T^{-1+\epsilon}\right\},
\end{equation}

\begin{equation}\label{N8 def}
N_9(l,T)=\left\{n\in N(l,T)\Bigl| 1<y_n, s_n=0\right\}.
\end{equation}

Furthermore, we will separately estimate the integrals over $y<1-T^{-B}$ and over $1-T^{-B}<y<1$.
\begin{lem}\label{lemma S1 est}
The following estimate holds
\begin{equation}\label{S1 est}
\sum_{|l|\ll T^{1+\epsilon}}\frac{S_{1}(l)}{|l|^2}\ll T^{-A}.
\end{equation}
\end{lem}
\begin{proof}
Since $s_n\gg T^{-1+\epsilon}$  we have $f_{-}(n/l,y)\gg T^{-2+\epsilon}$. Consequently, by Lemma \ref{lemma I estimate ynear1} 
\begin{equation}\label{int near1 small}
\int_{1-T^{-B}}^{1}
J_{-}\left(\frac{n}{l},y,T\right)\frac{dy}{(1-y^2)^{1/2}}\ll T^{-A}.
\end{equation}
Next, we consider the part of $S_{1}(l)$ with the integral over $y<1-T^{-B}.$ In this case $x_{-}(n/l,y)\gg T^{-2+\epsilon}$, and therefore, the contribution of this part to  $S_{1}(l)$ is $O(T^{-A})$ by \eqref{I integral estimate2}.
\end{proof}

\begin{lem}\label{lemma S2 est}
The following estimate holds
\begin{equation}\label{S2 est}
\sum_{|l|\ll T^{1+\epsilon}}\frac{S_{2}(l)}{|l|^2}\ll T^{-A}.
\end{equation}
\end{lem}
\begin{proof}
In this case $y_n=1$ and $0<|s_n|\ll T^{-1+\epsilon}$.
Using \eqref{yn def}, \eqref{sn def},  we obtain
\begin{equation}\label{nestimate1}
|n|\le|2l|\sqrt{1+O(T^{-2+\epsilon})}.
\end{equation}
It follows from \eqref{yn def} that $y_n=1$ is equivalent to
$ac+bd=a^2+b^2.$ The set of solutions of this linear equation is as follows
\begin{equation}\label{systemsolutions}
c_j=a+\frac{bj}{(a,b)},\quad
d_j=b-\frac{aj}{(a,b)}, \quad j\in\Z.
\end{equation}
Let $n_j=c_j+id_j.$ Then $s_{n_j}=-j/(a,b).$ Thus \eqref{f-to ynsn} can be rewritten as
\begin{equation*}
x_{-}(n_j/l,y)\gg1-y+\frac{s_{n_j}^2}{1-y}\gg |s_{n_j}|=\frac{|j|}{(a,b)}\gg\frac{1}{T^{1+\epsilon}}.
\end{equation*}
Furthermore, $|n_j|=|2l|\sqrt{1+j^2/(a,b)^2}$. Consequently, from \eqref{nestimate1} we conclude that $|j|\ll T^{\epsilon}.$ Therefore,  the contribution to $S_{2}(l)$ of the integral over  $y<1-T^{-B}$ for such $n_j$  is  negligible.

To deal with the integral over $1-T^{-B}<y<1$ we note that
\begin{equation*}
f_{-}(n_j/l,y)=(1-y)^2+s_{n_j}^2\gg \frac{j^2}{(a,b)^2}\gg\frac{1}{T^{2+\epsilon}}.
\end{equation*}
Applying Lemma \ref{lemma I estimate ynear1} we again obtain $O(T^{-A})$.
\end{proof}

\begin{lem}\label{lemma S3 est}
The following estimate holds
\begin{equation}\label{S3 est}
\sum_{|l|\ll T^{1+\epsilon}}\frac{S_{3}(l)}{|l|^2}\ll T^{2+4\theta+\epsilon}.
\end{equation}
\end{lem}
\begin{proof}
  
According to \eqref{sn def} the condition $s_n=0$ is equivalent to $ad-bc=0.$ The set of solutions of this linear equation is given by
\begin{equation}\label{systemsolutions2}
c_j=\frac{aj}{(a,b)},\quad
d_j=\frac{bj}{(a,b)}, \quad j\in\Z,\quad n_j=c_j+id_j.
\end{equation}
In this case, \eqref{f-to ynsn} can be rewritten as
\begin{equation}\label{x est1}
x_{-}(n_j/l,y)=\frac{(y-y_{n_j})^2}{1-y^2}, \quad
f_{-}(n_j/l,y)=(y-y_{n_j})^2.
\end{equation}
By  \eqref{yn def}  we have
\begin{equation}\label{ynj est1}
y_{n_j}=\frac{j}{(a,b)}, \quad 0\le j<(a,b)\Rightarrow
1-y_{n_j}\gg\frac{1}{T^{1+\epsilon}}.
\end{equation}
Applying Lemma \ref{lemma I estimate ynear1} and using \eqref{ynj est1} to estimate $f_{-}(n_j/l,y)$ from below, we obtain
\begin{equation*}
\int_{1-T^{-B}}^{1}
J_{-}\left(\frac{n_j}{l},y,T\right)\frac{dy}{(1-y^2)^{1/2}}\ll T^{-A}
\end{equation*}
if $B$ is sufficiently large. To estimate the remaining integral over $0<y<1-T^{-B}$ we decompose it into two parts. The first part is over
$$Y_1=\left\{y\,\Biggl|\,0<y<1-\frac{1}{T^{B}},\,\frac{(y-y_{n_j})^2}{1-y}\gg T^{-2+\epsilon}\right\}$$
and the second one is over
$$Y_2=\left\{y\,\Biggl|\,0<y<1-\frac{1}{T^{B}},\,\frac{(y-y_{n_j})^2}{1-y}\ll T^{-2+\epsilon}\right\}.$$
The integral over $Y_1$ is negligible due to \eqref{x est1} and \eqref{I integral estimate2}.  Using \eqref{I integral estimate3}, the integral over $Y_2$ can be bounded by
\begin{equation}\label{J est1}
\int_{Y_2}\frac{T^3dy}{1-y}\ll
\int_{|t|\ll T^{-1+\epsilon}\sqrt{1-y_{n_j}}}\frac{T^3dt}{1-y_{n_j}-t}\ll
\frac{T^{2+\epsilon}}{\sqrt{1-y_{n_j}}},
\end{equation}
where the estimate \eqref{ynj est1} was used. It remains to sum \eqref{J est1} over $n_j$ and over $|l|\ll T^{1+\epsilon}$  (see \eqref{decomposition 2}). As a result,
\begin{multline}\label{S(l) estimate2}
\sum_{|l|\ll T^{1+\epsilon}}\frac{S_{3}(l)}{|l|^2}\ll 
\sum_{|l|\ll T^{1+\epsilon}}
\sum_{j=1}^{(a,b)-1}\frac{|n_j^2-4l^2|^{2\theta}}{|l|^2}\frac{T^{2+\epsilon}}{\sqrt{1-y_{n_j}}}\\\ll
T^{2+\epsilon}\sum_{|l|\ll T^{1+\epsilon}}|l|^{4\theta-2}
\sum_{j=1}^{(a,b)-1}\left(1-\frac{j}{(a,b)}\right)^{2\theta-1/2}\\\ll
T^{2+\epsilon}\sum_{a^2+b^2\ll T^{2+\epsilon}}(a,b)(a^2+b^2)^{2\theta-1}
\ll T^{2+4\theta+\epsilon}.
\end{multline}
\end{proof}

\begin{lem}\label{lemma S4 est}
The following estimate holds
\begin{equation}\label{S4 est}
\sum_{|l|\ll T^{1+\epsilon}}\frac{S_{4}(l)}{|l|^2}\ll T^{3+4\theta+\epsilon}.
\end{equation}
\end{lem}
\begin{proof}
Since  $y_n<1-\epsilon_1$ for some $\epsilon_1>0$ it follows from  \eqref{f-to ynsn} that  for $1-T^{-B}<y<1$   we have $f_{-}(n/l,y)\gg \epsilon_1^2$. Therefore,
by  Lemma \ref{lemma I estimate ynear1}
\begin{equation*}
\int_{1-T^{-B}}^{1}
J_{-}\left(\frac{n}{l},y,T\right)\frac{dy}{(1-y^2)^{1/2}}\ll T^{-A}.
\end{equation*}
To estimate the remaining integral over $0<y<1-T^{-B}$ we decompose it into two parts. The first part is over $y$ such that $x_{-}(n/l,y)\gg T^{-2+\epsilon}$ and the second one is over
$$x_{-}(n/l,y)=\frac{(y-y_n)^2+s_n^2}{1-y^2}\ll T^{-2+\epsilon}.$$
The first part is negligible  due to
\eqref{I integral estimate2}.   Using \eqref{I integral estimate3} to estimate the second part and enlarging it to the domain $|y-y_n|\ll T^{-1+\epsilon}$  we obtain
\begin{equation}\label{J est2}
\int_{|y-y_n|\ll T^{-1+\epsilon}}\frac{T^3dy}{1-y}\ll T^{2+\epsilon}.
\end{equation}
Next, we sum \eqref{J est2} over $n$ and over $|l|\ll T^{1+\epsilon}$  (see \eqref{decomposition 2}).
Since $y_n<1-\epsilon_1$ and $0<|s_n|\ll T^{-1+\epsilon}$ it follows from  \eqref{yn def} and  \eqref{sn def} that $|n|<|2l|(1-\epsilon)$. Therefore,
\begin{equation}\label{S(l) estimate3}
\sum_{|l|\ll T^{1+\epsilon}}\frac{S_{4}(l)}{|l|^2}\ll \sum_{|l|\ll T^{1+\epsilon}}
\sum_{\substack{|n|<|2l|(1-\epsilon)\\0<|s_n|\ll T^{-1+\epsilon}}}
\frac{|n^2-4l^2|^{2\theta}}{|l|^2}T^{2+\epsilon}\ll
T^{2+\epsilon}\sum_{|l|\ll T^{1+\epsilon}}|l|^{4\theta-2}
\sum_{(c,d)\in\Omega_{a,b}}1,
\end{equation}
where (see \eqref{2l,n abcd} and \eqref{sn def})
\begin{equation*}
\Omega_{a,b}=\left\{(c,d)\Biggl|
c^2+d^2<(a^2+b^2)(1-\epsilon),
|ad-bc|<(a^2+b^2)T^{-1+\epsilon}
\right\}.
\end{equation*}
The sum over $(c,d)\in\Omega_{a,b}$ can be estimated  by the double integral over $\Omega_{a,b}$. To evaluate this integral we make the change of variables
\begin{equation*}
c=x\cos\vartheta_l-y\sin\vartheta_l,\quad
d=x\sin\vartheta_l+y\cos\vartheta_l, \quad \vartheta_l=\arg(2l),
\end{equation*}
and obtain
\begin{equation}\label{cd sum est1}
\sum_{(c,d)\in\Omega_{a,b}}1\ll
\iint_{\Omega_{|2l|}}1dxdy\ll|2l|^2T^{-1+\epsilon},
\end{equation}
where
\begin{equation*}
\Omega_{|2l|}=\left\{(x,y)\Biggl|
x^2+y^2<(1-\epsilon)|2l|^2,\,
|y|<|2l|T^{-1+\epsilon}
\right\}.
\end{equation*}
Substituting \eqref{cd sum est1} to \eqref{S(l) estimate3}, we infer
\begin{equation}\label{S(l) estimate4}
\sum_{|l|\ll T^{1+\epsilon}}\frac{S_{4}(l)}{|l|^2}\ll
T^{1+\epsilon}\sum_{|l|\ll T^{1+\epsilon}}|l|^{4\theta}\ll
T^{3+4\theta+\epsilon}.
\end{equation}
\end{proof}

\begin{lem}\label{lemma S5 est}
The following estimate holds
\begin{equation}\label{S5 est}
\sum_{|l|\ll T^{1+\epsilon}}\frac{S_{5}(l)}{|l|^2}\ll T^{-A}.
\end{equation}
\end{lem}
\begin{proof}
 Since
$s_n\gg T^{-2-\epsilon}$ using  \eqref{f-to ynsn} we have $f_{-}(n/l,y)\gg T^{-4-\epsilon}$. Therefore,
by  Lemma \ref{lemma I estimate ynear1} 
\begin{equation*}
\int_{1-T^{-B}}^{1}
J_{-}\left(\frac{n}{l},y,T\right)\frac{dy}{(1-y^2)^{1/2}}\ll T^{-A}.
\end{equation*}

Next, we are going to show that $s_n$ cannot be very small if $y_n$ is close to $1$, namely
\begin{equation}\label{|1-yn| small sn big}
1-T^{-2+\epsilon}<y_n<1\Rightarrow |s_n|\gg T^{-1-\epsilon}.
\end{equation}
Using  \eqref{yn def} and the fact that $0<1-y_n<T^{-2+\epsilon}$, we obtain
\begin{equation*}
0<\frac{a(a-c)+b(b-d)}{a^2+b^2}<T^{-2+\epsilon}.
\end{equation*}
This implies that $a^2+b^2>T^{2-\epsilon}.$ Furthermore, $a^2+b^2=|2l|^2\ll T^{2+\epsilon}$. Let $c_1=a-c$, $d_1=b-d$. Then it follows from the condition $1-T^{-2+\epsilon}<y_n<1$ that
\begin{equation}\label{system ab1}
T^{2-\epsilon}<a^2+b^2\ll T^{2+\epsilon},\quad 0<ac_1+bd_1\ll T^{\epsilon}.
\end{equation}
According to \eqref{sn def} 
$$s_n=\frac{bc_1-ad_1}{a^2+b^2}.$$
Note that since we assumed that $l$ belongs to the first quadrant, we have $a,b\ge0.$

Let $c_1d_1\ge0$. Thus  $c_1,d_1\ge0$  since $ac_1+bd_1>0.$  Suppose that both $a$ and $b$ are larger than
$T^{\epsilon}$. Since $ac_1+bd_1\ll T^{\epsilon}$ this implies that $c_1=d_1=0$, which is impossible since $n\neq2l.$ 
Therefore, without loss of generality, we assume that $a\gg T^{\epsilon}$ and $b\ll T^{\epsilon}$. It follows from the first double inequality in \eqref{system ab1} that  $T^{1-\epsilon}\ll a\ll T^{1+\epsilon}.$ Hence $c_1$ can only be $0$ (since $ac_1+bd_1\ll T^{\epsilon}$). Therefore,
$$|s_n|=\frac{ad_1}{a^2+b^2}\gg\frac{1}{a}\gg T^{-1-\epsilon}.$$

Let $c_1d_1\le0$. Suppose that $c_1\ge0$ and $d_1\le0$ (the other case can be treated similarly).
It follows from the second double inequality in \eqref{system ab1} that $c_1>b(-d_1)/a$. Therefore,
$$s_n>\frac{-d_1}{a}\gg T^{-1-\epsilon}.$$

Since $1-T^{-2+\epsilon}<y_n<1$ and $|s_n|\gg T^{-1-\epsilon}$ we conclude that for all $0<y<1$ we have
\begin{equation*}
x_{-}(n/l,y)=\frac{(y-y_n)^2+s_n^2}{1-y^2}\gg\frac{T^{\epsilon}}{T^2}.
\end{equation*}
Therefore, we can apply  \eqref{I integral estimate3} showing that the contribution of this case to $S_{5}(l)$ is negligibly small.

\end{proof}

\begin{lem}\label{lemma S6 est}
The following estimate holds
\begin{equation}\label{S6 est}
\sum_{|l|\ll T^{1+\epsilon}}\frac{S_{6}(l)}{|l|^2}\ll T^{3+4\theta+\epsilon}.
\end{equation}
\end{lem}
\begin{proof}
In this setting $1-\epsilon_1<y_n<1-T^{-2+\epsilon}$ and $0<|s_n|\ll T^{-1+\epsilon}$.
Here in contrast to the case when $y_n<1-\epsilon_1$ it is required to be more accurate in estimating the integral in \eqref{J est2}. In particular, $(1-y)\sim(1-y_n)$ can be very small (like $T^{-2-\epsilon}$).  Nevertheless, due to \eqref{ynsn notnear1} and since
$s_n\gg T^{-2-\epsilon}$ we have $f_{-}(n/l,y)\gg T^{-4-\epsilon}$ for $1-T^{-B}<y<1$  (see \eqref{f-to ynsn}). Therefore,
by  Lemma \ref{lemma I estimate ynear1}
\begin{equation*}
\int_{1-T^{-B}}^{1}
J_{-}\left(\frac{n}{l},y,T\right)\frac{dy}{(1-y^2)^{1/2}}\ll T^{-A}.
\end{equation*}
To estimate the remaining integral over $0<y<1-T^{-B}$ we decompose it into two parts. The first part is over
$$Y_1=\left\{y\,\Biggl|\,0<y<1-\frac{1}{T^{B}},\,\frac{(y-y_{n})^2+s_n^2}{1-y}\gg T^{-2+\epsilon}\right\}$$
and the second one is over
$$Y_2=\left\{y\,\Biggl|\,0<y<1-\frac{1}{T^{B}},\,\frac{(y-y_{n})^2+s_n^2}{1-y}\ll T^{-2+\epsilon}\right\}.$$
The integral over $Y_1$ is negligible due to \eqref{x est1} and \eqref{I integral estimate2}. To estimate the integral over $Y_2$ we first simplify the set $Y_2$. Since $1-\epsilon_1<y_n<1-T^{-2+\epsilon}$ we show that
$$Y_2\subset\left\{y\,\Biggl|\,0<y<1-\frac{1}{T^{B}},\, |y-y_{n}|\ll T^{-1+\epsilon}\sqrt{1-y_n}\right\}.$$
Applying \eqref{I integral estimate3}, we prove that the integral over $Y_2$ can be bounded by
\begin{equation}\label{J est3}
\int_{Y_2}\frac{T^3dy}{1-y}\ll
\int_{|t|\ll T^{-1+\epsilon}\sqrt{1-y_{n}}}\frac{T^3dt}{1-y_{n}-t}\ll
\frac{T^{2+\epsilon}}{\sqrt{1-y_{n}}}.
\end{equation}
It remains to sum \eqref{J est3} over $n$ and over $|l|\ll T^{1+\epsilon}$  (see \eqref{decomposition 2}).
Note that
$$|n|^2=|2l|^2(y_n^2+s_n^2)<|2l|^2(1+O(T^{-2+\epsilon})).$$
Therefore, we obtain
\begin{equation}\label{S(l) estimate5}
\sum_{|l|\ll T^{1+\epsilon}}\frac{S_{6}(l)}{|l|^2}\ll 
\sum_{|l|\ll T^{1+\epsilon}}
\sum_{\substack{|n|<|2l|(1+\epsilon),|s_n|\ll T^{-1+\epsilon}\\T^{-2+\epsilon}<1-y_n<\epsilon_1}}
\frac{|n^2-4l^2|^{2\theta}}{|l|^2}\frac{T^{2+\epsilon}}{\sqrt{1-y_{n}}}.
\end{equation}
As in the previous case, let  $c_1=a-c,\,d_1=b-d.$ Then
\begin{equation*}
1-y_n=\frac{ac_1+bd_1}{a^2+b^2},\quad s_n=\frac{bc_1-ad_1}{a^2+b^2},\quad
|2l-n|^2=c_1^2+d_1^2.
\end{equation*}
Moreover,
\begin{equation*}
|n|^2=|2l|^2+c_1^2+d_1^2-2(ac_1+bd_1).
\end{equation*}
The condition on $y_n$ in \eqref{S(l) estimate5} can be rewritten (and slightly simplified) as
\begin{equation*}
0<ac_1+bd_1\ll\epsilon_1|2l|^2.
\end{equation*}
The condition on $|n|$ in \eqref{S(l) estimate5} can be rewritten (and slightly simplified) as
\begin{equation*}
c_1^2+d_1^2\ll\epsilon|2l|^2.
\end{equation*}
So \eqref{S(l) estimate5} can be rewritten as
\begin{equation}\label{S(l) estimate6}
\sum_{|l|\ll T^{1+\epsilon}}\frac{S_{6}(l)}{|l|^2}\ll
T^{2+\epsilon}\sum_{|l|\ll T^{1+\epsilon}}|l|^{-1+2\theta}
\sum_{(c_1,d_1)\in\Omega(a,b)}\frac{(c_1^2+d_1^2)^{\theta}}{(ac_1+bd_1)^{1/2}},
\end{equation}
where
\begin{multline}\label{omega ab def}
\Omega(a,b)=\Biggl\{(c_1,d_1)\Biggl|
0<ac_1+bd_1\ll\epsilon_1|2l|^2,\quad
c_1^2+d_1^2\ll\epsilon|2l|^2,\,
|bc_1-ad_1|\ll|2l|^2T^{-1+\epsilon}
\Biggr\}.
\end{multline}
The sum over $(c_1,d_1)\in\Omega(a,b)$ can be estimated  by the double integral over $\Omega(a,b)$. To evaluate this integral we make the change of variables
\begin{equation*}
c_1=x\cos\vartheta_l-y\sin\vartheta_l,\quad
d_1=x\sin\vartheta_l+y\cos\vartheta_l, \quad \vartheta_l=\arg(2l).
\end{equation*}
Using the fact that $a=|2l|\cos\vartheta_l$ and  $b=|2l|\sin\vartheta_l$, we infer
\begin{equation}\label{c1d1 sum to int}
\sum_{(c_1,d_1)\in\Omega_{a,b}}\frac{(c_1^2+d_1^2)^{2\theta}}{(ac_1+bd_1)^{1/2}}\ll
\iint_{\Omega(|2l|)}\frac{(x^2+y^2)^{\theta}}{(x|2l|)^{1/2}}dxdy,
\end{equation}
where
\begin{equation*}
\Omega(|2l|)=\left\{(x,y)\Biggl|
x^2+y^2<\epsilon|2l|^2,\,
|y|\ll|2l|T^{-1+\epsilon},\,
x>0
\right\}.
\end{equation*}
Estimating the integrand $x^2+y^2$ by $|2l|^2$, we obtain
\begin{equation}\label{c1d1 sum to int2}
\sum_{(c_1,d_1)\in\Omega_{a,b}}\frac{(c_1^2+d_1^2)^{2\theta}}{(ac_1+bd_1)^{1/2}}\ll
|2l|^{2\theta-1/2}\int_{|y|\ll|2l|T^{-1+\epsilon}}\int_{0}^{|2l|}\frac{dxdy}{x^{1/2}}\ll
|2l|^{1+2\theta}T^{-1+\epsilon}.
\end{equation}
Substituting \eqref{c1d1 sum to int2} to \eqref{S(l) estimate6}, we conclude that
\begin{equation}\label{S(l) estimate7}
\sum_{|l|\ll T^{1+\epsilon}}\frac{S_{6}(l)}{|l|^2}\ll
T^{1+\epsilon}\sum_{|l|\ll T^{1+\epsilon}}|l|^{4\theta}\ll
T^{3+4\theta+\epsilon}.
\end{equation}
\end{proof}

\begin{lem}\label{lemma S7 est}
The following estimate holds
\begin{equation}\label{S7 est}
\sum_{|l|\ll T^{1+\epsilon}}\frac{S_{7}(l)}{|l|^2}\ll T^{-A}.
\end{equation}
\end{lem}
\begin{proof}
Since $y_n>1+T^{-2+\epsilon}$ and $0<|s_n|\ll T^{-1+\epsilon}$ we have
\begin{equation}\label{min f-}
\min_{0<y<1}f_{-}(n/l,y)=(y_n-1)^2+s_n^2\gg T^{-4-\epsilon},
\end{equation}
see \eqref{f-to ynsn} and \eqref{ynsn notnear1}. Thus for $1-T^{-B}<y<1$ we can apply Lemma \ref{lemma I estimate ynear1} showing that
\begin{equation*}
\int_{1-T^{-B}}^{1}
J_{-}\left(\frac{n}{l},y,T\right)\frac{dy}{(1-y^2)^{1/2}}\ll T^{-A}.
\end{equation*}
To estimate the remaining integral over $0<y<1-T^{-B}$ we are going to prove that $x_{-}(n/l,y)\gg T^{-2+\epsilon}$ (according to \eqref{I integral estimate2} this would imply that the integral is negligible).
For $1-T^{-2+\epsilon/2}<y<1$ the following estimate holds
\begin{equation}\label{x- y>1 est1}
\frac{(y-y_n)^2}{1-y^2}>\frac{T^{-4+2\epsilon}}{T^{-2+\epsilon/2}}=T^{-2+3\epsilon/2}.
\end{equation}
For $0<y<1-T^{-2+\epsilon/2}$ we have
\begin{equation}\label{x- y>1 est2}
\frac{(y-y_n)^2}{1-y^2}>\frac{(1+T^{-2+\epsilon}-y)^2}{1-y}>1-y>T^{-2+\epsilon/2}.
\end{equation}
Now the required estimate $x_{-}(n/l,y)\gg T^{-2+\epsilon}$  follows from \eqref{f-to ynsn}, \eqref{x- y>1 est1}, and \eqref{x- y>1 est2}.
\end{proof}

\begin{lem}\label{lemma S8 est}
The following estimate holds
\begin{equation}\label{S8 est}
\sum_{|l|\ll T^{1+\epsilon}}\frac{S_{8}(l)}{|l|^2}\ll T^{-A}.
\end{equation}
\end{lem}
\begin{proof}
As in Lemma \ref{lemma S5 est}  (see the proof of \eqref{|1-yn| small sn big}) we show that
\begin{equation}\label{|1-yn| small sn big2}
1<y_n<1+T^{-2+\epsilon}\,\Rightarrow\, |s_n|\gg T^{-1-\epsilon}.
\end{equation}
Similarly to the previous case, the contribution of the integral over $1-T^{-B}<y<1$ is negligible.  Furthermore, by \eqref{f-to ynsn} we have
\begin{equation*}
x_{-}(n/l,y)=\frac{(y_n-y)^2+s_n^2}{1-y^2}>\frac{(1-y)^2+s_n^2}{1-y}=1-y+\frac{s_n^2}{1-y}\gg s_n\gg T^{-1-\epsilon}.
\end{equation*}
Thus by  \eqref{I integral estimate2} the contribution of the integral over $0<y<1-T^{-B}$ is also negligible.

\end{proof}

\begin{lem}\label{lemma S9 est}
The following estimate holds
\begin{equation}\label{S9 est}
\sum_{|l|\ll T^{1+\epsilon}}\frac{S_{9}(l)}{|l|^2}\ll T^{-A}.
\end{equation}
\end{lem}
\begin{proof}
As in Lemma \ref{lemma S3 est}  we obtain that 
\begin{equation}\label{|ynj-1| big2}
y_{n}-1\gg\frac{1}{T^{1+\epsilon}},
\end{equation}
and therefore, we conclude that the contribution of $1-T^{-B}<y<1$ is negligibly small.  Next,  using \eqref{f-to ynsn} and \eqref{|ynj-1| big2} we obtain 
\begin{equation*}
x_{-}(n/l,y)=\frac{(y_n-y)^2+s_n^2}{1-y^2}>\frac{((1-y)+(y_n-1))^2}{1-y}>1-y+(y_n-1)\gg T^{-1-\epsilon}.
\end{equation*}
Finally, by  \eqref{I integral estimate2} the contribution of the integral over $0<y<1-T^{-B}$ is negligibly small.
\end{proof}
\section{Proof of Theorem \ref{thm:main}}\label{PGT}
 Assume that for some $\alpha>0$
\begin{equation}
\sum_{T<r_j<2T}\alpha_{j}|L(\sym^2 u_{j},s)|^2 \ll T^{3+2\alpha}.
\end{equation}
Then following the arguments on \cite[p. 5363]{BCCFL} we obtain
\begin{equation}
\sum_{0<r_j\leq T}X^{ir_j}\ll T^{(7+2\alpha)/4+\epsilon}X^{1/4+\epsilon}+T^{2}.
\end{equation}
Using this estimate and following the proof of \cite[Corollary 1.4]{BBCL} we conclude that the error term for  $\pi_{\Gamma}(X)$ is bounded by
\begin{equation}
O\left(\frac{X^{2+\alpha/2+\epsilon}}{Y^{3/4+\alpha/2}}+X^{(4\theta+6)/5+\epsilon}Y^{2/5}+\frac{X^2}{Y}+X^{1+\epsilon}\right).
\end{equation}
Choosing 
\begin{equation}
Y=X^{\frac{10\alpha+16(1-\theta)}{23+10\alpha}},
\end{equation}
we have
\begin{equation}
O(X^{\frac{3}{2}+\frac{\alpha(2+16\theta)+24\theta-1}{46+20 \alpha}+\epsilon}).
\end{equation}
Theorem \ref{thm:2moment}  allows us to take $\alpha=2\theta$ which yields Theorem  \ref{thm:main}.

\section*{Acknowledgments}
Research of Olga Balkanova was funded by RFBR, project number 19-31-60029.

\end{document}